\documentclass[a4paper, 12pt, oneside]{amsart}

\usepackage[T1]{fontenc}
\usepackage{microtype}

\usepackage{amsmath, amsthm, amssymb, amsfonts, mathtools}
\usepackage{mathrsfs}
\usepackage{hyperref}
\usepackage{enumerate}
\usepackage{graphicx}
\usepackage{paralist}
\usepackage{pinlabel}

\usepackage[marginparwidth=2cm]{geometry}
\usepackage[scaled=0.875]{helvet}
\usepackage{courier}

\usepackage[maxbibnames=99,maxalphanames=99,giveninits=true,style=alphabetic,citestyle=alphabetic,doi=false,isbn=false,url=false]{biblatex}
\AtEveryBibitem{\clearlist{language}}

\theoremstyle{plain}
	\newtheorem{thm}{Theorem}
	\newtheorem{lem}[thm]{Lemma}
	\newtheorem{cor}[thm]{Corollary}
	\newtheorem{prop}[thm]{Proposition}
	\newtheorem{fact}[thm]{Fact}

\theoremstyle{definition}
	\newtheorem{defn}[thm]{Definition}
	
    \newtheorem{conj}[thm]{Conjecture}

\theoremstyle{remark}
	\newtheorem{rmk}[thm]{Remark}

\numberwithin{thm}{section}
\numberwithin{equation}{section}

\newcommand{\1}{\mathrm{Id}}
\DeclareMathOperator{\rk}{rk}
\DeclarePairedDelimiter\gen{\langle}{\rangle}
\DeclarePairedDelimiter\CR{[}{]}
\DeclarePairedDelimiter\card{|}{|}
\DeclarePairedDelimiter\abs\lvert\rvert
\newcommand{\A}{\mathbb{A}}

\newcommand{\N}{\mathbb{N}}
\newcommand{\Z}{\mathbb{Z}}

\newcommand{\R}{\mathbb{R}}

\newcommand{\K}{\mathbb{K}}
\renewcommand{\H}{\mathbb{H}}

\let\P\relax
\DeclareMathOperator{\P}{\mathbb{P}}
\DeclareMathOperator{\GL}{GL}
\DeclareMathOperator{\SL}{SL}
\DeclareMathOperator{\SLpm}{SL^\pm}

\DeclareMathOperator{\PO}{PO}

\DeclareMathOperator{\PGL}{PGL}

\DeclareMathOperator{\Aut}{Aut}

\DeclareMathOperator{\Int}{int}
\DeclareMathOperator{\relInt}{int_\text{rel}}

\DeclareMathOperator{\Span}{Span}
\DeclareMathOperator{\Facets}{\Sigma}

\DeclareMathOperator{\Conv}{Conv}
\newcommand{\join}{\mathbin{\otimes}}

\renewcommand{\setminus}{\smallsetminus}

\newcommand{\Thm}{Thm.~}
\newcommand{\ie}{\emph{i.e.}~}
\newcommand{\eg}{\emph{e.g.}~}
\newcommand{\Prop}{Prop.~}
\newcommand{\Lm}{Lm.~}
\newcommand{\Cor}{Cor.~}
\newcommand{\resp}{resp.~}

\newcommand{\Ex}{Ex.~}

\title{Projective reflection groups of finite covolume}

\date{\today}

\author[B. Fl\'echelles]{Balthazar Fl\'echelles}
\address{Institut Fourier, 100 rue des mathématiques, 38610, Gières, France}
\email{balthazar.flechelles@univ-grenoble-alpes.fr}

\author[S. Hwang]{Seunghoon Hwang}
\address{Department of Mathematical Sciences, Seoul National University, Seoul 08826, South Korea}
\email{seunghoon1013@snu.ac.kr}

\begin{document}

\begin{abstract}
	We show that the Coxeter polytopes that have finite volume in their Vinberg domains are exactly the quasiperfect Coxeter polytopes of negative type, \ie the Coxeter polytopes that are contained in their properly convex Vinberg domain, at the exception of some vertices that are $C^1$ points of the boundary.
    As a corollary, we show that for reflection groups \emph{à la Vinberg}, the Vinberg domain is the only invariant properly convex domain if and only if the action is of finite covolume on the Vinberg domain and the dimension is at least $2$.
\end{abstract}

\maketitle

\tableofcontents

\section{Introduction} \label{sec:intro}

In this article, we characterize the reflection groups \emph{\`a la Vinberg} that act with finite covolume on a properly convex domain, \ie an open, convex and bounded subset $\Omega$ of an affine chart of a real projective space $\P(V)$ of dimension $d\geq 1$. One can define a metric $d_\Omega$ on $\Omega$, called the \emph{Hilbert metric}, that is invariant under the \emph{automorphisms} of $\Omega$, \ie the linear transformations that leave $\Omega$ invariant. Our interest lies in the study of \emph{convex projective orbifolds}, which are the quotients of properly convex domains by discrete subgroups of their automorphism groups.
Many symmetric spaces admit convex projective structures, such as the real hyperbolic spaces, Euclidean spaces, and the symmetric spaces of $\SL(n,\K)$ for $n\geq 3$ and $\K$ the field of the real, complex, or quaternionic numbers (note however that the Hilbert metric differs from that of the symmetric space, except for the real hyperbolic spaces).

The Hilbert metric $d_\Omega$ of a properly convex domain $\Omega$ induces a Finsler metric and a Borel measure on $\Omega$, both invariant under the automorphisms of $\Omega$. This allows one to ask whether a discrete subgroup $\Gamma$ of the automorphism group acts with finite covolume. If this is the case, we say that $\Gamma$ \emph{quasidivides} $\Omega$, and that $\Omega$ is \emph{quasidivisible}.
If in fact $\Gamma$ acts cocompactly on $\Omega$, we say that $\Gamma$ \emph{divides} $\Omega$ and that $\Omega$ is \emph{divisible}. It is not difficult to see that (cocompact) lattices of the Lie groups whose symmetric space admit a convex projective structure easily provide examples of (quasi)divisible domains (such examples are called \emph{symmetric}).

The first genuinely new examples of divisible domains in \cite{kacVinberg1967} used reflection groups in dimension $2$, and the theory was later greatly generalized in \cite{vinberg1971}. A \emph{reflection group} is a discrete subgroup $\Gamma$ of $\SLpm(V)$ that is generated by linear reflections along hyperplanes of $V$; it is abstractly isomorphic to a Coxeter group.
Vinberg gave in \cite{vinberg1971} necessary and sufficient conditions on the reflections for $\Gamma$ to be a reflection group, and showed that an infinite reflection group always tiles a convex open subset of $\P(V)$, called the \emph{Vinberg domain}. When the Vinberg domain is properly convex, a reflection group can be seen as the group generated by chosen reflections along the facets of a properly convex polytope $P\subset\P(V)$. The data of a properly convex polytope and reflections along its facets satisfying Vinberg's conditions is called a \emph{Coxeter polytope}. The Vinberg domain $\Omega_P$ of a Coxeter polytope $P$ does not have to be properly convex, but it is in most cases. When it is the case, we say that $P$ is \emph{of negative type}. 

In the same article, Vinberg proved that the reflection group $\Gamma_P$ of a Coxeter polytope of negative type $P$ divides its Vinberg domain $\Omega_P$ if and only if $P$ is \emph{perfect}, \ie for every vertex $v$ of $P$, the subgroup $\Gamma_v$ of $\Gamma_P$ generated by the reflections along the facets of $P$ containing $v$ is finite. He also showed that when $\Omega_P$ is symmetric, $\Gamma_P$ quasidivides $\Omega_P$ if and only if $P$ is \emph{quasiperfect}, \ie for all vertices $v$ of $P$, the group $\Gamma_v$ is either finite or conjugated to a maximal rank parabolic subgroup of $\PO(d,1)$.

A similar condition was later introduced in \cite{marquis2017}, as a regularity hypothesis: a Coxeter polytope $P$ is \emph{$2$-perfect} if, for all vertices $v$ of $P$, the subgroup $\Gamma_v$ acts as a perfect reflection group on the space $V/\Span v$ (note that $v\in\P(V)$ defines a $\Gamma_v$-invariant line $\Span v$ in $V$). Observe that perfect Coxeter polytopes are quasiperfect, and that quasiperfect Coxeter polytopes are $2$-perfect. Marquis showed in \cite{marquis2017} that if $P$ is a $2$-perfect Coxeter polytope of negative type, then $\Gamma_P$ quasidivides $\Omega_P$ if and only if $P$ is quasiperfect. Following a recent endeavor to investigate Vinberg theory without such regularity hypotheses (see \cite{dancigerGueritaudKasselLeeMarquis2025,audibertDoubaLeeMarquis2025}), we generalize this result with minimal assumptions.

\begin{thm}[see Theorem \ref{thm:main}]\label{thm:mainIntro}
   Let $P$ be a Coxeter polytope of negative type. Then $\Gamma_P$ quasidivides $\Omega_P$ if and only if $P$ is quasiperfect.
\end{thm}

Prior to our work, the theoretic study of divisible domains had made significant progress in the sixties and the early two thousands (see \cite{benzecri1960,koszul1968,vey1970,benoist2004,benoist2003,benoist2005,benoist2006}). It appears that there are two types of divisible domains exhibiting different behaviors: either $\Omega$ is \emph{strictly convex}, meaning that its boundary contains no non-trivial segments, in which case the group $\Gamma$ dividing $\Omega$ is Gromov hyperbolic, or $\Omega$ is not strictly convex. In the strictly convex case, there exist examples that are non-symmetric divisible deformations of symmetric divisible domains \cite{goldman1990,johnsonMillson1987,benoist2000} for $d\geq 2$ and that are not \cite{benoist2006QI,kapovich2007} for $d\geq 4$. In the non-strictly convex case, there are examples for $d\geq 3$ as well \cite{benoist2006,ballasDancigerLee2018,choiLeeMarquis2020,choiLeeMarquis2022,leeMarquisRiolo2022,blayacViaggi2025}.

The general question of studying quasidivisible domains is less advanced. In the strictly convex case, \cite{cooperLongTillmann2015,cramponMarquis2014} thoroughly studied the geometry of quasi\-divisible domains, and showed that the holonomy of the ends of a finite volume strictly convex orbifold of dimension $d$ are always conjugated to a maximal rank parabolic subgroup of $\PO(d,1)$: we say that such orbifolds have \emph{hyperbolic ends}. Examples of quasidivisible strictly convex domains were produced for all $d\geq 2$ in \cite{marquis2012a,marquis2012b,marquis2016,marquis2017,ballas2014,ballasMarquis2020}.

In the non-strictly convex case however, it turns out finite volume orbifolds may have different types of ends, called \emph{generalized cusps}, which were defined and studied in \cite{cooperLongTillmann2018,ballasCooperLeitner2020,ballasCooperLeitner2022}. It is still possible to have only hyperbolic cusps, as \cite{marquis2017,choiLeeMarquis2020,choiLeeMarquis2022} show for $d \in\{3,4,5,6,7,8\}$ using reflection groups, but we have examples for $d\geq 3$ when we allow for generalized cusps \cite{ballas2015,bobb2019,ballasMarquis2020,ballas2021}.

\begin{rmk}
    When $P$ is a quasiperfect Coxeter polytope of negative type, $\Omega_P/\Gamma_P$ has hyperbolic ends. Hence, Theorem \ref{thm:mainIntro} shows that one cannot find examples of quasidivisible domains whose quotient presents generalized cusps (that are not hyperbolic cusps) using linear reflection groups.
\end{rmk}

Recent works in convex projective geometry have shown that in many geometrically interesting cases, there are infinitely many different invariant properly convex domains (see for instance \cite{cooperLongTillmann2018,dancigerGueritaudKassel2024,flechellesIslamZhu2026}). However, it was known by \cite{vey1970,benoist2000} that in most cases, if a group $\Gamma$ divides a properly convex domain $\Omega\subset\P(V)$, then $\Omega$ is the unique $\Gamma$-invariant properly convex domain in $\P(V)$.
The question remains open in the quasidivisible case, though Marquis had shown in \cite{marquis2017} that when $P\subset\P(V)$ is a $2$-perfect Coxeter polytope of negative type, $\Gamma_P$ preserves a unique properly convex domain in $\P(V)$ if and only if $P$ is quasiperfect. Theorem \ref{thm:mainIntro} allows us to extend this theorem of Marquis by dropping the ``$2$-perfectness'' assumption.

\begin{thm}\label{thm:uniqueDomain}
   Let $P\subset\P(V)$ be a Coxeter polytope of negative type. Then $\Gamma_P$ preserves a unique properly convex domain in $\P(V)$ if and only if $P$ is quasiperfect and $\dim P\geq 2$.
\end{thm}

More generally, we provide a characterization of the Coxeter polytopes of negative type for which the convex hull of the proximal limit set coincides with the Vinberg domain (see Theorem \ref{thm:characVinDomainEqualsCHofLimSet}). This discussion is linked to the question of whether or not the proximal limit set of quasidivisible domains fills their boundaries, as we explain at the end of Section \ref{sec:uniqueDomain}. Indeed, Theorem \ref{thm:uniqueDomain} implies that if a Coxeter polytope $P$ of negative type is such that the proximal limit set of $\Gamma_P$ equals the boundary of its Vinberg domain $\Omega_P$, then $P$ is quasiperfect and $\Gamma_P$ is not virtually abelian (see Corollary \ref{cor:ifProxLimSetFillsThenQPLarge}). We conjecture that the reciprocal is true (see Conjecture~\ref{conj:characProxLimSetFilling}).

\subsection*{Organization of the paper}

In Section \ref{sec:ConvProjGeom}, we give reminders about convex projective geometry, and in Sections \ref{sec:VinTheory} and \ref{sec:CoxPol}, we list the definitions and tools we will need from Vinberg theory. In Section \ref{sec:lemmas}, we prove elementary results necessary for the proof of our main theorems. We then give in Sections \ref{sec:proofMainThm} and \ref{sec:uniqueDomain} the proofs of Theorems \ref{thm:mainIntro} and \ref{thm:uniqueDomain} respectively.

\subsection*{Acknowledgments}

We would like to thank Gye-Seon Lee for his outstanding support and for suggesting us to work together on this problem. We are also grateful to Yosuke Morita and Sami Douba for helpful discussions.

The authors acknowledge the support of the Institut Henri Poincar\'e (UAR 839 CNRS-Sorbonne Universit\'e) and LabEx CARMIN (ANR-10-LABX-59-01). The first author was supported by the ANR-23-CE40-0012 HilbertXfield and the National Research Foundation of Korea (NRF) grant funded by the Korea government (MSIT) (No. RS-2023-00252171), and received partial funding from the grant Partenariat Hubert Curien 50166PH and the European Research Council (ERC) under the European Union’s Horizon 2020 research and innovation program (ERC starting grant DiGGeS, grant agreement No 715982, and ERC consolidator grant GeometricStructures, grant agreement No 614733). The second author was supported by Samsung Science and Technology Foundation under Project Number SSTF-BA2001-03 and the National Research Foundation of Korea(NRF) grant funded by the Korea government(MSIT) (No. RS-2023-00259480). 

\section{Convex projective geometry}\label{sec:ConvProjGeom}

Let $d\geq 1$ be a fixed integer, and $V$ be a real vector space of dimension $d+1$.

\subsection{Convex subsets of $\P(V)$} \label{sec:convex}

Recall that a cone $C\subset V$ is a subset that is invariant under homotheties. We say that a convex cone $C$ is \emph{sharp} if it does not contain a full line.

\begin{defn}
    A subset of $\P(V)$ is said to be \emph{(properly) convex} if it admits a lift with respect to the projection $\P:V\setminus\{0\}\to\P(V)$ that is a (sharp) convex cone.
    
    A \emph{properly convex domain} is an open subset of $\P(V)$ that is properly convex.
\end{defn}

Recall that in an affine space $\A$, the convex hull of a subset $X$ is the set of all convex combinations of points of $X$, and it is the smallest convex subset of $\A$ containing $X$. In projective space, there is a priori no good notion of convex hull since there are two projective segments between any two points. However, notice that if $\Omega\subset\P(V)$ is a properly convex domain, then for any pair of distinct points $x,y\in\overline{\Omega}$, there is a unique projective segment contained in $\overline{\Omega}$ joining $x$ and $y$, corresponding to the segment between $x$ and $y$ in any affine chart of $\P(V)$ containing $\Omega$. It follows that for any subset $X$ of $\overline{\Omega}$, the convex hull of $X$ in an affine chart $\A$ containing $\Omega$ does not depend on the choice of $\A$, so we may define $\Conv(X)$ to be the convex hull of $X$ in any affine chart containing $\Omega$.

Unless stated otherwise, the convex hull of a subset $X$ of the closure of a properly convex domain refers to this definition. If we need to take a convex hull of a subset of $\P(V)$ that is not a priori contained in the closure of a properly convex domain, we will specify an affine chart within which we take the convex hull.

\subsection{Faces of a convex set} \label{sec:polytopes}

Let $C$ be a convex subset of $\P(V)$. Recall that the \emph{span} of $C$, denoted by $\Span C$, is the smallest subspace of $V$ such that $C\subset\P(\Span C)$. In fact, $C$ has non-empty relative interior in $\P(\Span C)$. The \emph{dimension} of $C$ is the dimension of $\P(\Span C)$.

\begin{defn}\label{def:faces}
    Let $C$ be a closed convex subset of $\P(V)$. The \emph{faces} of $C$ are the equivalence classes of the relation $\sim$, where for any $x,y\in C$, $x\sim y$ if and only if there is an open interval in $C$ containing both $x$ and $y$.
\end{defn}

Observe that a closed convex subset of $\P(V)$ is the union of its faces. A priori, it may have uncountably many faces (\eg an ellipsoid). We order the faces of a closed convex subset by the inclusion of the closures: $f$ is smaller than $g$ if $\overline f\subset\overline g$.

\begin{defn} \label{defn:polytope}
    A \emph{projective polytope} is a properly convex subset with non-empty interior that has finitely many faces. Equivalently, a projective polytope is the projectivization of a sharp convex cone with non-empty interior that is the intersection of finitely many closed half-spaces.
\end{defn}

Note that the interior of a projective polytope $P\subset\P(V)$ is its unique face of dimension $d$. The faces of $P$ of dimension $d - 1$, $d - 2$, or $0$ are respectively called \emph{facets}, \emph{ridges} and \emph{vertices} of $P$. We will also consider $\varnothing$ to be a face of $P$ of dimension~$-1$.

\subsection{The Hilbert metric}

Let $\Omega\subset \P(V)$ be a properly convex domain. We will denote by $\CR{a:b:c:d}$ the cross-ratio of $4$ aligned points $a,b,c,d\in\P(V)$, with the normalization $\CR{0:1:\lambda:\infty} = \lambda$ for all $\lambda\in\P(\R^2)$. The \emph{Hilbert metric} $d_\Omega$ on $\Omega$ is given by the formula
\begin{equation*}
    d_\Omega(x,y)\coloneqq \frac12\log\CR{x':x:y:y'}
\end{equation*}
for $x\neq y\in\Omega$, where $x'$ and $y'$ are the intersections of $\partial\Omega$ with the projective line through $x$ and $y$, in such a way that $x',x,y,y'$ are aligned in that order (see Figure~\ref{fig:hilbertmetric}).

\begin{figure}[ht]
	\centering
    \labellist \small\hair 2pt
        \pinlabel{$x'$} [u] at 35 108
        \pinlabel{$x$} [u] at 107 137
        \pinlabel{$y$} [u] at 163 163
        \pinlabel{$y'$} [u] at 232 197
        \pinlabel{$\Omega$} [u] at 235 100
    \endlabellist

	\includegraphics[scale=.5]{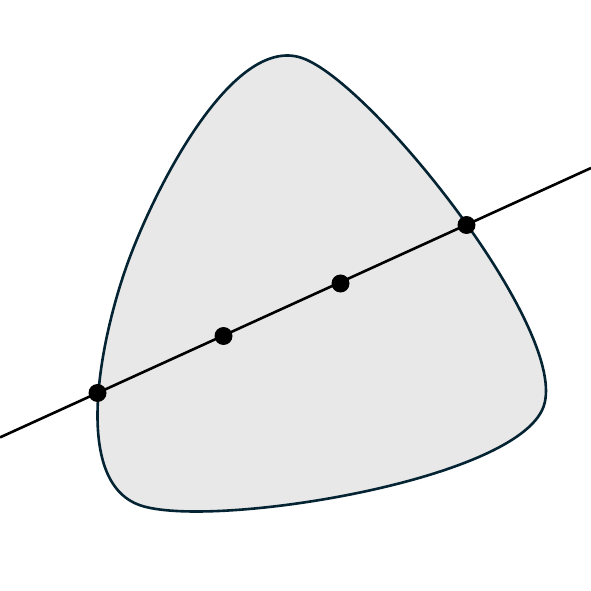}
    \caption{Hilbert metric}
    \label{fig:hilbertmetric}
\end{figure}

Since the definition of $d_\Omega$ depends only on cross-ratios and $\Omega$ itself, the elements of $\PGL(V)$ that preserve $\Omega$ are isometries of $(\Omega,d_\Omega)$, called \emph{automorphisms}.
For convenience, we prefer to work with subgroups of $$\SLpm(V)\coloneqq \{g\in\GL(V)\mid \det g = \pm 1\},$$ so we define \emph{the automorphism group} of $\Omega$ by
\begin{equation*}
    \Aut(\Omega) \coloneqq  \{g\in \SLpm(V)\mid  gC_\Omega = C_\Omega\}
\end{equation*}
where $C_\Omega\subset V$ is any sharp convex cone lifting $\Omega$.

The following are some of the interesting features of the metric space $(\Omega,d_\Omega)$.
\begin{enumerate}
	\item It is a complete proper geodesic metric space whose topology is the topology induced from $\P(V)$.
	\item Projective line segments contained in $\Omega$ are geodesics for $d_\Omega$, but in general, there may be other geodesics.
	\item A subgroup of $\Aut(\Omega)$ acts properly discontinuously on $\Omega$ if and only if it is a discrete subgroup of $\SL^\pm(V)$.
\end{enumerate}
As an example, when $\Omega$ is an ellipsoid in some affine chart of $\P(V)$, $(\Omega,d_\Omega)$ is isometric to $\H^d$.

The Hilbert metric $d_\Omega$ induces a Finsler metric defined by
\begin{equation*}
    F_\Omega(x,v)\coloneqq\left.\frac d{dt}\right|_{t=0}d_\Omega(x,x+tv)
\end{equation*}
for all point $x\in\Omega$ and vector $v$ tangent to $\Omega$ at $x$. This Finsler structure allows to define several $\Aut(\Omega)$-invariant Borel measures that are absolutely continuous with respect to the Lebesgue measure on any affine chart containing $\Omega$ (see \cite[\S 5.5.3]{buragoBuragoIvanov2001}). This may seem to be a problem for us, but it turns out that by \cite[\Thm 5.5.18]{buragoBuragoIvanov2001}, a Borel subset of $\Omega$ has finite measure for one of them if and only if it does for all of them. All of these measures satisfy (it is a consequence of the monotonicity assumption):

\begin{fact} \label{fact:measure}
	Let $\Omega_1,\Omega_2$ be properly convex domains of $\P(V)$ such that $\Omega_1\subset\Omega_2$. Then $\mu_{\Omega_2}(B)\leq\mu_{\Omega_1}(B)$ for any Borel subset $B\subset\Omega_1$.
\end{fact}

We will work here with the Busemann measure $\mu_\Omega$ on $\Omega$, which coincides with the $d$-dimensional Hausdorff measure of $(\Omega,d_\Omega)$ (see \cite[\Ex 5.5.13]{buragoBuragoIvanov2001}). A proof of the above proposition for the Busemann measure can be found in \cite[\Prop 5.(iv)]{colboisVernicosVerovic2004}.

\subsection{The proximal limit set} \label{sec:prox}

An element $\gamma\in\SLpm(V)$ is said to be \emph{proximal} if the maximum $\lambda(\gamma)$ of the moduli of the complex eigenvalues of $\gamma$ is an eigenvalue of multiplicity $1$. The $\lambda(\gamma)$-eigenspace of $\gamma$ defines a point in $\P(V)$, which we call the \emph{attracting fixed point} of $\gamma$.

\begin{defn}
    A discrete subgroup $\Gamma$ of $\SLpm(V)$ is said to be \emph{proximal} if it contains a proximal element. In this case, the \emph{proximal limit set} of $\Gamma$ is the closure $\Lambda_\Gamma$ of the set of attracting fixed points of the proximal elements of $\Gamma$.
\end{defn}

A subgroup $G$ of $\SLpm(V)$ is \emph{irreducible} if it does not preserve any nontrivial subspace of $V$, and \emph{strongly irreducibly} if every finite-index subgroup of $G$ is irreducible.

\begin{fact} [{\cite[\Prop 3.1]{benoist2000}}] \label{fact:limitset}
    Let $\Gamma$ be a discrete subgroup of $\SLpm(V)$ preserving a properly convex domain $\Omega\subset\P(V)$. If $\Gamma$ is irreducible, then $\Gamma$ is proximal. In this case, we have:
    \begin{enumerate}[(i)]
        \item There exists a largest $\Gamma$-invariant properly convex domain $\Omega_{\rm max}$ of $\P(V)$ containing $\Omega$. In other words, $\Omega_{\rm max}$ contains every $\Gamma$-invariant properly convex domain of $\P(V)$ that contains $\Omega$.
        \item $\Lambda_\Gamma$ is contained in every non-empty $\Gamma$-invariant closed subset of $\P(V)$. Moreover, $\Omega_{\rm min}\coloneqq \Int\Conv(\Lambda_\Gamma)$ is the smallest $\Gamma$-invariant properly convex domain of $\P(V)$ contained in $\Omega$. In other words, it is contained in every $\Gamma$-invariant properly convex domain of $\P(V)$ contained in $\Omega$.
    \end{enumerate}
    If moreover $\Gamma$ is strongly irreducible, then $\Omega_{\rm max}$ (\resp $\Omega_{\rm min}$) is the largest (resp. smallest) $\Gamma$-invariant properly convex domain of $\P(V)$. In other words, for any $\Gamma$-invariant properly convex domain $\Omega'\subset\P(V)$, we have $\Omega_{\rm min}\subset\Omega'\subset\Omega_{\rm max}$.
\end{fact}

\section{Vinberg theory}\label{sec:VinTheory}

Vinberg theory provides a method for building discrete representations of Coxeter groups as projective reflection groups. We give here the definitions and notations we will use throughout the paper.

\subsection{Coxeter groups} \label{sec:Coxgrps}

A \emph{Coxeter system} is a pair $(S,M)$ of a non-empty finite set $S$ and a symmetric matrix $M=(m_{st})_{s,t\in S}$ on $S$ with entries in $\N\cup\{\infty\}$ such that $m_{st}=1$ if and only if $s=t$ for all $s,t\in S$.

\begin{defn} \label{defn:Coxgrp}
	The \emph{Coxeter group} $W_{S,M}$ is given by the following presentation:
	\begin{equation*}
	    W_{S,M}=\gen{s\in S\mid (st)^{m_{st}}=1, s,t\in S, m_{st}\neq\infty}
	\end{equation*}
	A \emph{standard subgroup} of $W=W_{S,M}$ is the subgroup $W_T$ of $W$ generated by a subset $T$ of $S$.
\end{defn}

Since $m_{ss}=1$ for each $s\in S$, each generator $s\in S$ is an involution. Note that $s\neq t$ in $S$ commute if and only if $m_{st}=2$. We say that two subsets $T_1$ and $T_2$ of $S$ are \emph{orthogonal} if $T_2\subset T_1^\perp$ (or equivalently $T_1\subset T_2^\perp$), where
\begin{equation*}
    T_1^\perp \coloneqq  \{s\in S\mid \forall t\in T_1,m_{st} = 2\}\subset S
\end{equation*}
Observe that $T_1$ and $T_2$ are orthogonal if and only if $W_{T_1\cup T_2} = W_{T_1}\times W_{T_2}$.

\begin{defn}
    An \emph{irreducible component} of $S$ is a minimal non-empty subset $T$ of $S$ such that $T$ is orthogonal to its complement $S\setminus T$.
\end{defn}

Observe that $W$ is the direct product of the standard subgroups associated to the irreducible components of $S$. We say that $W$ is \emph{irreducible} if it does not decompose as a direct product of non-trivial standard subgroups, or in other words, if $S$ has a unique irreducible component. $W$ is said to be \emph{reducible} if it is not irreducible.

\begin{defn} \label{defn:Gram}
	Let $W=W_{S,M}$ be a Coxeter group. The \emph{Gram matrix} of $W$, is the symmetric matrix $G_W\coloneqq (-2\cos(\pi/m_{st}))_{s,t\in S}$.
\end{defn}

\begin{defn}
	An irreducible Coxeter group is \emph{spherical} (\resp \emph{affine}) if its Gram matrix is positive-definite (\resp positive-semidefinite and not positive-definite). It is \emph{large} if it has a finite index subgroup that admits a non-abelian free quotient.
\end{defn}

\begin{fact}[\cite{coxeter1932, coxeter1934,margulisVinberg2000}] \label{fact:irredCoxgrp}
	Let $W$ be an irreducible Coxeter group. Then $W$ is either spherical, affine or large. Moreover,
	\begin{enumerate}[(i)]
		\item if $W$ is spherical, then $W$ is finite;
		\item if $W$ is affine, then $W$ is virtually isomorphic to $\Z^{\card{S}-1}$ and $\card{S}\geq 2$.
	\end{enumerate}
\end{fact}

We say that a reducible Coxeter group is spherical, affine or large if all of its irreducible components are so.

\begin{rmk} \label{rmk:irredCoxgrp}
	Suppose $W$ is an irreducible Coxeter group that is either spherical or affine. Then by \cite[\Prop 11]{vinberg1971}, all proper standard subgroups of $W$ are spherical.
\end{rmk}

\subsection{Vinberg's theorem} \label{sec:Vinthry}

A \emph{projective reflection} of $\P(V)$ is an involution of $\SLpm(V)$ which is the identity on a hyperplane of $V$. Given a projective reflection $\sigma$, we can choose a point $v\in V$ (called a \emph{polar} of $\sigma$) and a linear form $\alpha\in V^\ast$ such that $v$ spans the $(-1)$-eigenspace of $\sigma$ and $\ker\alpha$ is the fixed hyperplane of $\sigma$. If we further impose that $\alpha(v) = 2$, we see that $\sigma = \1 - \alpha\otimes v$, that is $\sigma(x) = x - \alpha(x)v$ for all $x\in V$.

Let $S$ be a finite set. We pick a projective reflection $\sigma_s = \1 - \alpha_s\otimes v_s\in\SLpm(V)$ for each $s\in S$, where $(\alpha_s,v_s)\in V^\ast\times V$ satisfies $\alpha_s(v_s) = 2$, and we let
\begin{equation*}
    \Delta \coloneqq  \bigcap_{s\in S}\{v\in V\mid\alpha_s(v) \leq 0\}
\end{equation*}
Assume furthermore that we chose $(\sigma_s)_{s\in S}$ so that $\Delta$ has non-empty interior. 

We say that the group $\Gamma<\SLpm(V)$ generated by the reflections $(\sigma_s)_{s\in S}$ is a \emph{reflection group} if, for all $\gamma\in\Gamma\setminus\{1\}$,
\begin{equation*}
	\Int(\Delta)\cap\gamma\cdot\Int(\Delta)=\varnothing
\end{equation*}
If this is the case, we say that $\Delta$ is the \emph{fundamental cone} of $\Gamma$.

Using Definition \ref{def:faces} in $V$ instead of $\P(V)$, we can define the faces of $\Delta$. For any face $f$ of $\Delta$, we let $S_f \coloneqq  \{s\in S\mid f\subset\ker\alpha_s\}$, and we denote by $\Gamma_f$ the subgroup of $\Gamma$ generated by the reflections $(\sigma_s)_{s\in S_f}$.

We now want to give necessary and sufficient conditions on a collection $(\sigma_s)_{s\in S}$ of reflections for the resulting group $\Gamma$ to be a reflection group. This is the content of Vinberg's theorem, but we first need to define Cartan matrices.

\begin{defn} \label{defn:Cartan}
	A \emph{Cartan matrix} on a finite set $S$ is a matrix $A=(A_{st})_{s,t\in S}$ satisfying:
	\begin{enumerate}
		\item for all $s\in S$, $A_{ss}=2$;
		\item for all $s\neq t\in S$, $A_{st}\leq 0$;\label{item:Cconditions}
		\item for all $s\neq t\in S$, $A_{st}=0$ if and only if $A_{ts}=0$;
		\item for all $s\neq t\in S$, $A_{st}A_{ts}\geq 4$ or $A_{st}A_{ts}=4\cos^2(\pi/k)$ for some integer $k\geq 2$.
	\end{enumerate}	
\end{defn}

Given a Cartan matrix $A$ on $S$, we let $M =(m_{st})_{s,t\in S}$, where
\begin{equation*}
	m_{st}=
	\begin{cases}
	    1 & \text{if $s=t$,}\\
		k & \text{if $A_{st}A_{ts}=4\cos^2(\pi/k)$,} \\
		\infty & \text{if $A_{st}A_{ts}\geq 4$}
	\end{cases}
\end{equation*}
for all $s,t\in S$. Then $(S,M)$ is a Coxeter system, so the Cartan matrix $A$ defines a Coxeter group $W_A\coloneqq W_{S,M}$. Notice that for any Coxeter group $W$, the Gram matrix $G_W$ of $W$ is a Cartan matrix that satisfies $W_{G_W} = W$.

\begin{fact} [{\cite[\Thm 1 \& 2]{vinberg1971}}] \label{fact:VinbergThm}
	Let $S$ be a finite set, and $(\sigma_s)_{s\in S}$ be projective reflections in $\SLpm(V)$ such that $\Delta$ has non-empty interior. Then $\Gamma$ is a reflection group if and only if $A\coloneqq(\alpha_s(v_t))_{s,t\in S}$ is a Cartan matrix. In this case:
	\begin{enumerate}[(i)]
		\item $s\mapsto\sigma_s$ defines a discrete and faithful representation $\rho:W_A\to\GL(V)$ with image $\Gamma$.
		\item $C \coloneqq  \Int\bigcup_{\gamma\in\Gamma}\gamma\Delta$ is a non-empty open convex cone of $V$.
		\item $\Gamma$ acts properly discontinuously on $C$ with fundamental domain $\Delta\cap C$.
		\item A face $f$ of $\Delta$ lies in $C$ if and only if $\Gamma_f$ is finite; otherwise $f\subset\partial C$.
	\end{enumerate}
\end{fact}

If $\Gamma$ is a reflection group, then the representation $\rho: W_A\to\SLpm(V)$ given by Fact \ref{fact:VinbergThm} is called a \emph{Vinberg representation}, and we say that $A$ is the Cartan matrix of $\Gamma$.

\begin{rmk} \label{rmk:Tits}
    Given a Cartan matrix $A$ on a finite set $S$, if we let $(\alpha_s)_{s\in s}$ denote the dual of the canonical basis of $\R^S$ and $(v_s)_{s\in S}$ the columns of $A$ seen as points of $\R^S$, then $A = (\alpha_s(v_t))_{s,t\in S}$. Therefore, we can always build a Vinberg representation $\rho_A: W_A\to\SLpm(\R^S)$ with $\Delta$ the (sharp) convex cone on the canonical simplex $\Conv(e_s)_{s\in S}$. Observe that this is also a corollary of \cite[\Thm 5]{vinberg1971} using the characteristic $\{A,L_c(A),0,0\}$.
    
    Applying this construction to the Gram matrix of a Coxeter group $W$ yields a Vinberg representation of $W$, called the \emph{Tits representation} of $W$.
\end{rmk}

\subsection{Cartan matrices}

If $A$ is a Cartan matrix on a finite set $S$, then for any subset $T$ of $S$, we let $A_T\coloneqq (A_{st})_{s,t\in T}$. It is also a Cartan matrix. Moreover, the Coxeter group $W_A$ allows to define the relation $\perp$ on $S$, so we can define as before the irreducible components of $S$. In an ordering of $S$ compatible with the decomposition of $S$ into its irreducible components $S_1\sqcup \dots\sqcup S_k$, the matrix $A$ splits into a block diagonal matrix whose blocks are the Cartan matrices $A_{S_i}$. We call these Cartan submatrices of $A$ the \emph{irreducible components} of $A$.

A Cartan matrix $A$ is \emph{irreducible} if it has a single irreducible component, or equivalently, if the associated Coxeter group $W_A$ is irreducible. If $A$ is an irreducible Cartan matrix, then $2I-A$ is irreducible with nonnegative entries, so by the Perron--Frobenius theorem, its complex eigenvalue of largest modulus is real, positive and simple. In other words, $A$ admits a real simple eigenvalue $\lambda_A$ such that $|2-\mu|$ is maximal at $2-\lambda_A$, where $\mu$ is any complex eigenvalue of $A$.

\begin{defn}
    We say that an irreducible Cartan matrix $A$ is of \emph{positive}, \emph{zero}, or \emph{negative type} following the sign of its Perron--Frobenius eigenvalue $\lambda_A$.
\end{defn}

If $A$ is \emph{reducible} (\ie not irreducible), we say it is of \emph{positive} (\resp \emph{zero}, \resp \emph{negative}) \emph{type} if all the irreducible components of $A$ are of positive (\resp zero, \resp negative) type.

\begin{fact} [{\cite[\Thm 3]{vinberg1971}}] \label{fact:CartanMatrixCharac}
	Let $X=(X_s)_{s\in S}\in\R^S$ be a column vector, and $A$ be an irreducible Cartan matrix. When $X$ has nonnegative (\resp positive) entries, we write $X\geq0$ (\resp $X>0$).
	\begin{enumerate}[(i)]
		\item If $A$ is of positive type, then $A$ is nonsingular, and $AX\geq0$ implies either $X=0$ or $X>0$.
		\item If $A$ is of zero type, then $AX\geq0$ implies $AX=0$.
		\item If $A$ is of negative type, then $X\geq0$ and $AX\geq0$ imply $X=0$.
	\end{enumerate}
\end{fact}

\begin{rmk}\label{rmk:typeCartanMatrixGivesWitness}
    Let $A$ be a (non necessarily irreducible) Cartan matrix on $S$ of positive (\resp zero, \resp negative) type. Then there is a vector $X>0$ in $\R^S$ such that $AX > 0$ (\resp $AX = 0$, \resp $AX < 0$). That is because we can split $S$ into its irreducible components $S_1,\dots,S_k$ and choose for each $i$ a vector $X_i>0$ in $\R^{S_i}$ such that $A_{S_i}X_i$ has the right sign, by Perron--Frobenius theorem. Since $A$ is the block diagonal matrix with blocks $(A_{S_i})_{i=1}^k$, we can then choose $X = (X_s)_{s\in S}$ where $X_s = (X_i)_s$ if $s\in S_i$.
\end{rmk}

\subsection{Reduced and dual-reduced Vinberg representations}

Let $\rho:W\rightarrow\GL(V)$ be a Vinberg representation with Cartan matrix $A$ and image $\Gamma$. Define 
\begin{align*}
V_\alpha &\coloneqq\bigcap_{s\in S}\ker\alpha_s & V_v &\coloneqq\Span\{v_s\mid s\in S\} & V_v^\alpha &\coloneqq V_v/(V_\alpha\cap V_v)
\end{align*}
Note that $V_\alpha\neq V$ and $V_v\neq\{0\}$ are $\Gamma$-invariant subspaces of $V$. We say that $\rho$ is \emph{reduced} (\resp \emph{dual-reduced}) if $V_\alpha=\{0\}$ (\resp $V_v=V$).

The following was proven in \cite{vinberg1971} (see Equation (32) and the corollary to Proposition 19). The addition when $W$ is large is from \cite[\Thm 2.18]{marquis2017}.
\begin{fact}\label{fact:irredIffRedAndDualRed}
    Let $\rho: W\to\SLpm(V)$ be a Vinberg representation with Cartan matrix $A$, and suppose $W$ is irreducible.
    Then $\rho$ is irreducible if and only if $\rho$ is reduced and dual-reduced, if and only if $\rk A = d+1$.
    
    If moreover $W$ is large, then $\rho$ is also strongly irreducible.
\end{fact}

We will also need the following fact which allows to locate the proximal limit set.
\begin{fact}[{\cite[\Cor 3.25]{dancigerGueritaudKasselLeeMarquis2025}}] \label{fact:limSetInSpanPolars}
    Let $\rho: W\to\SLpm(V)$ be a Vinberg representation with a Cartan matrix of negative type.
    Then $\rho(W)$ is proximal and $\Lambda_{\rho(W)}\subset\P(V_v)$.
\end{fact}

\section{Coxeter polytopes}\label{sec:CoxPol}

When using Vinberg theory for the purpose of producing examples in convex projective geometry, it is interesting to suppose that the fundamental cone $\Delta$ of our reflection group is sharp. By \cite[\Prop 18]{vinberg1971}, this is equivalent to supposing that the Vinberg representation $\rho: W\to\SLpm(V)$ is reduced. If this is the case, $P\coloneqq \P(\Delta)$ is a projective polytope whose facets are spanned by the hyperplanes $\P(\ker\alpha_s)$ for $s\in S$.

\begin{defn}
    Let $P\subset\P(V)$ be a projective polytope, and $S$ denote the collection of its facets. A \emph{Coxeter polytope structure} on $P$ is a collection of pairs $(\alpha_s,v_s)_{s\in S}\in (V\times V^\ast)^S$ such that
    \begin{enumerate}
        \item the matrix $A_P\coloneqq (\alpha_s(v_t))_{s,t\in S}$ is a Cartan matrix on $S$;
        \item the convex cone $\Delta \coloneqq  \bigcap_{s\in S}\{\alpha_s\leq 0\}$ is a lift of $P$ in $V$.
    \end{enumerate}
    
    A \emph{Coxeter polytope} is a projective polytope $P$ equipped with a Coxeter polytope structure. Most of the time, the Coxeter polytope structure will not be made explicit, and will be denoted by $(\alpha_s,v_s)_{s\in S}$ as above when needed, unless stated otherwise.
\end{defn}

Since the data of a Coxeter polytope $P$ is enough to define the reflections $(\sigma_s)_{s\in S}$, there is an associated Cartan matrix $A_P$, Coxeter group $W_P$, reflection group $\Gamma_P$, and Vinberg representation $\rho_P:W_P\to\SLpm(V)$. By Fact \ref{fact:VinbergThm}, $\Gamma_P$ preserves an open convex cone $C_P$ tiled by $C_P\cap\Delta$, so the \emph{Vinberg domain}
\begin{equation*}
    \Omega_P \coloneqq  \Int\bigcup_{\gamma\in\Gamma_P}\gamma P = \P(C_P)
\end{equation*}
is a convex open subset of $\P(V)$. We will also denote by $\Lambda_P$ the proximal limit set of $\Gamma_P$ when it is proximal.

\subsection{Types of Coxeter polytopes} \label{sec:types}

We can associate types to Coxeter polytopes by means of their Cartan matrix.

\begin{defn} \label{defn:polytopetype}
	Let $P\subset\P(V)$ be a Coxeter polytope. We say $P$ is
	\begin{enumerate}
		\item \emph{elliptic} if $A_P$ is of positive type;
		\item \emph{of zero type} if $A_P$ is of zero type, and \emph{parabolic} if moreover $\rk A_P = d$;
		\item \emph{of negative type} if $A_P$ is of negative type, and \emph{loxodromic} if  moreover $\rk A_P = d+1$;
	\end{enumerate}
	Similarly, we say that $P$ is \emph{irreducible} or \emph{reducible} if $A_P$ (or equivalently, $W_P$) is so.
\end{defn}

If $A_P$ is of positive type, then we have $\rk A_P=d+1$ by Fact \ref{fact:CartanMatrixCharac}. One can produce examples of non-parabolic Coxeter polytopes of zero type and non-loxodromic Coxeter polytopes of negative type by Remark \ref{rmk:Tits} using examples of such Cartan matrices.

The following result gives a necessary and sufficient condition for the Vinberg domain to be properly convex.

\begin{fact} [{\cite[\Prop 25]{vinberg1971}}] \label{fact:negTypeIffVinDomPropConv}
	Let $P$ be a Coxeter polytope. Then $\Omega_P$ is properly convex if and only if $P$ is of negative type.
\end{fact}

\begin{defn}
    Let $P$ be a Coxeter polytope of negative type. We say that $P$ is a \emph{finite volume Coxeter polytope} if the volume of $P$ in its Vinberg domain is finite.
\end{defn}

We will also need the following result, which can be inferred from Vinberg's results.
\begin{fact}[{\cite[\Prop 2.9]{marquis2017}}]\label{fact:negTypeImpliesLargeOrATilda}
    Let $P\subset\P(V)$ be a Coxeter polytope. If $P$ is of negative type, either $W_P$ is affine of type $\tilde{A}_d$, or it is large.

    If $W_P$ is large, then $P$ is of negative type.
\end{fact}

\subsection{Indecomposable Coxeter polytopes}\label{par:decPolytopes}
Let $C_1\subset V_1$ and $C_2\subset V_2$ be two sharp open convex cones in the finite-dimensional real vector spaces $V_1$ and $V_2$. The sum $C_1 + C_2$ is a sharp open convex cone of $V_1\oplus V_2$.
Let now $\Omega_1 \coloneqq  \P(C_1)$ and $\Omega_2\coloneqq \P(C_2)$. A \emph{join} of $\Omega_1$ and $\Omega_2$ is a properly convex domain of $\P(V_1\oplus V_2)$ of the form $\P(C_1 + C_2)$. Notice that since there are two different sharp open convex cones lifting $\Omega_1$ and $\Omega_2$ (the cones and their opposite), there are $2$ different joins. More generally, there are $2^{r-1}$ different joins of $r$ properly convex domains. Observe that all the joins are image of one another by elements of $\SLpm(V)$, so it does not usually make much difference to work with one of them rather than the others. Moreover, since the subgroup of $\SLpm(V_1)$ preserving $C_1$ also preserves $-C_1$ and the same holds for $C_2$, it is clear that all the joins of $\Omega_1$ and $\Omega_2$ are invariant under exactly the same elements of $\SLpm(V)$.

One can also interpret a join of $\Omega_1$ and $\Omega_2$ as the interior of the convex hull of $\Omega_1$ and $\Omega_2$ in \emph{some} affine chart. The choice of the affine chart is what makes the join not unique. Observe that $\Omega_1$ and $\Omega_2$ are always faces of their joins. A properly convex domain is said to be \emph{decomposable} if it is a join of some of its faces, and \emph{indecomposable} otherwise.

Let $P\subset\P(V_1)$ and $Q\subset\P(V_2)$ be two Coxeter polytopes with facets $S_1$ and $S_2$, and Coxeter structures $(\alpha_s,v_s)_{s\in S_1}$ and $(\beta_s,w_s)_{s\in S_2}$ respectively. Observe that the closure of any join of $P$ and $Q$ is a projective polytope in $\P(V_1\oplus V_2)$ whose facets are spanned by the projective hyperplanes $\P(V_1\oplus\Span s_2)$ for $s_2\in S_2$ and $\P(\Span s_1\oplus V_2)$ for $s_1\in S_1$. Hence, the facets of any join of $P$ and $Q$ are in bijection with $S_1\cup S_2$, and we can define a Coxeter structure $(\alpha_s,v_s)_{s\in S_1}\cup (\beta_s,w_s)_{s\in S_2}$ on the polytope
\begin{equation*}
    P\join Q\coloneqq \P\left(\bigcap_{s_1\in S_1}\{\alpha_{s_1}\leq 0\} \cap \bigcap_{s_2\in S_2}\{\beta_{s_2}\leq 0\}\right)
\end{equation*}
where the linear forms $(\alpha_s)_{s\in S_1}$ and $(\beta_s)_{s\in S_2}$ have been canonically extended to $V_1\oplus V_2$. Observe that the Coxeter polytope structures of $P$ and $Q$, because they select a unique preferred lift of these polytopes as sharp convex cones, also select a \emph{unique} join of $P$ and $Q$, given in the last equation. We check that since $w_s\in V_2\subset\ker\alpha_t$ for all $(s,t)\in S_2\times S_1$, we indeed get a Cartan matrix $A_{P\join Q}$ that is block-diagonal with blocks $A_{P}$ and $A_{Q}$, so $P\join Q$ is indeed a Coxeter polytope. It is then elementary to check that $\Omega_{P\join Q}$ is the unique join of $\Omega_P$ and $\Omega_Q$ that contains $P\join Q$, and that $\Gamma_{P\join Q} = \Gamma_P\times \Gamma_Q$, with $\Gamma_P$ acting trivially on $V_2$ and naturally on $V_1$, and similarly for $\Gamma_Q$.

Observe that though the Coxeter polytope structures on $P$ and $Q$ allow to define a unique join of $P$ and $Q$ as a Coxeter polytope, it is still possible to put a Coxeter polytope structure on the other joins of $P$ and $Q$ by changing the signs of the Coxeter polytopes structures on $P$ or $Q$ (for instance, by changing $(\alpha_s,v_s)_{s\in S_1}$ into $(-\alpha_s,-v_s)_{s\in S_1}$). This discussion also generalizes to the join of an arbitrary number of Coxeter polytopes.

\begin{defn} \label{defn:decomposable}
    A Coxeter polytope $P$ is \emph{decomposable} if it can be written as the join of smaller Coxeter polytopes, and \emph{indecomposable} otherwise.
\end{defn}

Note that an irreducible Coxeter polytope is necessarily indecomposable. 
It will be useful to describe the proximal limit set of a join of Coxeter polytopes.

\begin{lem}\label{lm:limSetOfDecCoxPol}
   Let $P$ be a Coxeter polytope that is a join of Coxeter polytopes $(P_i)_{i=1}^k$ of negative type. Then $\Gamma_P$ is proximal and
   \begin{equation*}
      \Lambda_P = \bigcup_{i=1}^k\Lambda_{P_i}
   \end{equation*}
\end{lem}
\begin{proof}
   Let $V = \bigoplus_{i=1}^k V_i$ be the decomposition of $V$ so that each $P_i$ is a Coxeter polytope in $\P(V_i)$. Observe that $\Gamma_P$ is isomorphic to $\Gamma_{P_1}\times\dots\times\Gamma_{P_k}$, with $\Gamma_{P_i}$ acting naturally on $\Omega_{P_i}$ and trivially on the other factors $\Omega_{P_j}$ for $j\neq i$. In particular, a proximal element in $\Gamma_P$ is of the form $(\gamma_1,\dots,\gamma_k)$ where $\gamma_i\in\Gamma_{P_i}$ is proximal for some $i$ with top eigenvalue larger than the modulus of the complex eigenvalues of the $\gamma_j$ for $j\neq i$.
   There exist proximal elements in $\Gamma_P$ as each $\Gamma_{P_i}$ is proximal by Fact \ref{fact:limSetInSpanPolars}, so $\Gamma_P$ is proximal. Moreover, by the description of the proximal elements of $\Gamma_P$ above, $\Lambda_P = \bigcup_{i=1}^k\Lambda_{P_i}$.
\end{proof}

\subsection{Links of Coxeter polytopes} \label{sec:links}

Let $P\subset\P(V)$ be a Coxeter polytope and $f$ be a face of $P$. Recall from Fact \ref{fact:VinbergThm} that the stabilizer of $f$ in $\Gamma_P$ is $\Gamma_f\coloneqq \rho_P(W_{S_f})$ where $S_f=\{s\in S\mid f\subset\P(\ker\alpha_s)\}$. The restriction of $\rho_P$ to $W_{S_f}$ is a Vinberg representation with Cartan matrix $A_{S_f}$, which is not reduced as $\bigcap_{s\in S_f}\ker\alpha_s = \Span f$ is non-trivial. However, this representation induces a reduced Vinberg representation $W_{S_f}\to\SLpm(V/\Span f)$. The associated Coxeter polytope is the \emph{link} $P_f\subset\P(V/\Span f)$ of $P$ at $f$, given by the projection in $\P(V/\Span f)$ of the convex cone $\bigcap_{s\in S_f}\{\alpha_s\leq 0\}$.

We will qualify a face $f$ of $P$ by an adjective that we defined for a Coxeter polytope if the link $P_f$ is qualified by this adjective. For instance, a face is parabolic, or of negative type, if and only if $P_f$ is parabolic or of negative type.

Let $P$ be a Coxeter polytope. By Fact \ref{fact:VinbergThm}, one can see that $\Gamma_P$ acts cocompactly on $\Omega_P$ if and only if $P\subset\Omega_P$, if and only if all its vertices lie in $\Omega_P$, if and only if all its vertices are elliptic. This motivates the following definition.

\begin{defn} \label{defn:perfect}
    Let $P$ be a Coxeter polytope. We say that $P$ is \emph{perfect} if all the vertices of $P$ are elliptic.
\end{defn}

Elliptic and parabolic Coxeter polytopes are always perfect by \cite[\Prop 11]{vinberg1971}.
Vinberg also slightly weakened the notion of a perfect Coxeter polytope in order to accommodate the reflections groups in $\H^d$ that act with finite covolume.

\begin{defn} \label{defn:quasiperfect}
    Let $P$ be a Coxeter polytope. We say that $P$ is \emph{quasiperfect} if all the vertices of $P$ are either elliptic, or parabolic.
\end{defn}

We will need the following regularity result. It is a corollary of a classification by Vinberg of the quasiperfect Coxeter polytopes.

\begin{fact}[{\cite[\Prop 26]{vinberg1971}}] \label{fact:QPcharacterization}
    Let $P$ be a quasiperfect Coxeter polytope of negative type. Then $P$ is loxodromic irreducible.
\end{fact}

\subsection{Maximality of the Vinberg domain}

The Vinberg domain is in many cases the largest invariant properly convex domain.

\begin{fact}[{\cite[\Prop 4.1]{dancigerGueritaudKasselLeeMarquis2025}}] \label{fact:DGKLMmaxDomain}
    Let $P\subset\P(V)$ be an irreducible Coxeter polytope of negative type such that
    $\card{S}\geq 3$. Then $\Omega_P$ is the largest $\Gamma_P$-invariant properly convex domain of $\P(V)$. In other words, $\Omega_P$ contains any $\Gamma_P$-invariant properly convex domain of $\P(V)$.
\end{fact}

Danciger--Guéritaud--Kassel--Lee--Marquis also gave a description of the smallest invariant properly convex domain defined in Fact \ref{fact:limitset}.

\begin{fact}[{\cite[\Thm 5.2]{dancigerGueritaudKasselLeeMarquis2025}}]\label{fact:DGKLMminDomain}
    Suppose $P\subset\P(V)$ is an irreducible Coxeter polytope of negative type such that $V_v = V$. Then 
    \begin{equation*}
        \Omega_{\rm min} = \Int\bigcup_{\gamma\in \Gamma_P}\gamma\cdot \left(P\cap\P(\Conv(v_s)_{s\in S})\right)
    \end{equation*}
    is the smallest $\Gamma_P$-invariant properly convex domain of $\P(V)$ contained in $\Omega_P$. In other words, $\Omega_{\rm min}$ is contained in all $\Gamma_P$-invariant properly convex domains of $\P(V)$ contained in $\Omega_P$.
\end{fact}

\begin{rmk}\label{rmk:largeDGKLMminDomain}
    In particular, if $\Gamma_P$ is strongly irreducible, then $\Omega_\text{min}$ also equals $\Int\Conv(\Lambda_P)$ and is the smallest $\Gamma_P$-invariant properly convex domain of $\P(V)$ by Fact \ref{fact:limitset}.
    This holds in particular when $\Gamma_P$ is a large Coxeter group, by Fact \ref{fact:irredIffRedAndDualRed}.
\end{rmk}

\section{Lemmas}\label{sec:lemmas}

We prove here general facts that are key to the proof of Theorem \ref{thm:mainIntro}.

\subsection{No proper faces of negative type}

The following is a generalization of \cite[\Lm 4.25]{marquis2017}.

\begin{lem} \label{thm:decomposableinfvol}
   Let $V\coloneqq  E_1\oplus E_2$ be a non-trivial decomposition, and $\Omega_i\subset\P(E_i)$ be properly convex domains for $(i=1,2)$. Let $\Omega\coloneqq\Omega_1\join\Omega_2$ be a join of $\Omega_1$ and $\Omega_2$. If $P\subset\Omega$ is convex with non-empty interior, and $\Omega_1\cap\overline P$ has non-empty interior in $\Omega_1$, then $\mu_\Omega(P)=\infty$. 
\end{lem}
\begin{proof}
    Let $\pi:\P(V)\setminus \P(E_1)\to\P(E_2)$ be the projectivization of the projection onto $E_2$ parallel to $E_1$, and $Q_1 \subset\relInt(\Omega_1\cap\overline{P})$ and $Q_2\subset\relInt\pi(P)$ be closed subsets of non-empty interior. Lastly, let $Q\coloneqq P\cap \Conv(Q_1\cup Q_2)$, and define $h\in \GL(V)$ by
	\begin{equation*}
		h(v_1+v_2)=2v_1+v_2
	\end{equation*}
        for all $(v_1,v_2)\in V=E_1\oplus E_2$. Observe that $\frac1{\abs{\det h}}h$ defines an automorphism of $\Omega$ that preserves all projective lines through a point of $\P(E_1)$ and a point of $\P(E_2)$, moving points on these lines closer to $\P(E_1)$. Since $Q$ is foliated by intervals of such projective lines with one endpoint on $Q_1\subset\P(E_1)$, we have $h(Q)\subset Q$. By shrinking $P$ if necessary, we may suppose that $\overline{P}\cap \P(E_2) = \varnothing$. In this case, $Q\setminus h(Q)$ contains a neighborhood of the union of the other endpoint of these intervals, so it has non-empty interior. It follows that $\mu_\Omega(Q) = \mu_\Omega(h(Q)) + \mu_\Omega(Q\setminus h(Q)) > \mu_\Omega(Q)$, so we have $\mu_\Omega(P)\geq\mu_\Omega(Q) = +\infty$. 
\end{proof}

By Lemma \ref{thm:decomposableinfvol}, we have

\begin{cor} \label{cor:finvolimplies}
	If $P$ is a Coxeter polytope of negative type with finite volume, then $P$ has no proper faces of negative type.
\end{cor}

\begin{proof}
	Suppose $P$ has a proper face $f$ of negative type.
	The Vinberg domain $\Omega_f$ of the link $P_f$ is a properly convex domain of $\P(V/\Span f)$ by Fact \ref{fact:negTypeIffVinDomPropConv}. Let $E$ be a supplement of $\Span f$ in $V$ such that $\P(E)$ does not meet the closure of $\Omega_P$. We have an isomorphism $\phi:\P(V/\Span f)\to\P(E)$. Choose a properly convex domain $\Omega'\subset\P(\Span f)$ such that $\Omega'$ contains $f$. Now if we let $\Omega\coloneqq\Omega'\join\phi(\Omega_f)$, then $\Omega$ contains $\Omega_P$, and Fact \ref{fact:measure} and Lemma \ref{thm:decomposableinfvol} yield $\mu_{\Omega_P}(P)\geq\mu_\Omega(P)=+\infty$.
\end{proof}

\subsection{Finding a bigger face}

Recall the following definition from \cite{vinberg1971}.

\begin{defn} \label{def:definesAFace}
   Let $P\subset\P(V)$ be a Coxeter polytope with facet set $S$.
   We say that a subset $S'$ of $S$ \emph{defines a face of $P$} if one of the equivalent conditions below is satisfied:
   \begin{enumerate}
      \item there is a face $f$ of $P$ (potentially $\varnothing$ or $\Int P$) such that $S' = S_f$;
      \item the following system has a solution in $V$:
          \begin{equation*}
             \begin{cases}
                \alpha_s = 0 &\text{for $s\in S'$}\\
                \alpha_s < 0 &\text{for $s\not\in S'$}
             \end{cases}
          \end{equation*}
          \label{item:definesAFaceSystem}
      \item for any $X=(X_s)_{s\in S}\in\R^S$ such that $X_s\geq 0$ for all $s\not\in S'$, $\sum_{s\in S} X_s\alpha_s=0$ implies that $X_s = 0$ for all $s\not\in S'$.
   \end{enumerate}
\end{defn}

Given a Cartan matrix $A=(A_{st})_{s,t\in S}$ on $S$ and a subset $T$ of $S$, recall that $A_T\coloneqq (A_{st})_{s,t\in T}$ is also a Cartan matrix. Like for the faces of $P$, we qualify $T$ by an adjective if this adjective applies to $A_T$. We let $T^+$ (\resp $T^0$, \resp $T^-$) be the union of the irreducible components of $T$ that are of positive (\resp zero, \resp negative) type. The following lemma is inspired by \cite[\Thm 4]{vinberg1971}.

\begin{lem} \label{lm:biggerFace}
	Let $P$ be a Coxeter polytope with facet set $S$ and $T_1,T_2\subset S$ disjoint. If $T_1\cup T_2$ defines a face of $P$ and $T_1\subset T_2^\perp$, then $T_1\cup T_2^0$, $T_1\cup T_2^0\cup T_2^+$, and $T_1\cup T_2^0\cup T_2^-$ also define a face of $P$.
\end{lem}

\begin{proof}
	Let $T\in\{T_2^0, T_2^0\cup T_2^+, T_2^0\cup T_2^-\}$. Consider a row vector $X=(X_s)_{s\in S}\in \R^S$ such that $X_s\geq 0$ for all $s\not\in T_1\cup T$ and $\sum_{s\in S} X_s\alpha_s = 0$. We want to show that $X_s = 0$ for all $s\not\in T_1\cup T$. Since $T_1\cup T_2$ defines a face of $P$ and $X_s\geq 0$ holds for all $s\not\in T_1\cup T_2$, we have $X_s = 0$ for all $s\not\in T_1\cup T_2$. There only remains to show that $X_s=0$ for all $s\in T'\coloneqq T_2\setminus T$, which is a union of irreducible components of $T_2^+\cup T_2^-$.

	Let $A\coloneqq A_P$, and for any $S'\subset S$, let $X_{S'}\coloneqq (X_s)_{s\in S'}$. Note that we have $X_{T_1\cup T_2}A_{T_1\cup T_2}=0$ from the previous paragraph. This decomposes into $X_{T_1}A_{T_1}=0$ and $X_{T_2}A_{T_2}=0$, as $T_1$ and $T_2$ are orthogonal, and similarly, the second identity yields $X_TA_T=0$ and $X_{T'}A_{T'}=0$. In particular, we have $X_IA_I=0$ for every irreducible component $I$ of $T'$, but since $A_I$ is either of positive or negative type, Fact \ref{fact:CartanMatrixCharac} applied to the transpose of $A_I$ yields $X_I=0$. Hence $X_{T'}=0$.
\end{proof}

\subsection{Maximal boundary faces}

A \emph{boundary face} of a Coxeter polytope of negative type $P$ is a face of $P$ that lies in $\partial\Omega_P$. A \emph{maximal boundary face} of $P$ is a boundary face of $P$ that is maximal.
Recall from Fact \ref{fact:VinbergThm} that a face $f$ of $P$ is a boundary face if and only if the standard subgroup $\Gamma_f$ is infinite.

\begin{prop} \label{prop:maxbdryface}
    Let $P$ be a Coxeter polytope of negative type.
	A boundary face $f$ of $P$ is maximal if and only if $f$ is perfect. In particular, the maximal boundary faces of $P$ are exactly its parabolic and perfect loxodromic faces.
\end{prop}

\begin{proof}
	Suppose $f$ is maximal and $P_f$ has a vertex $v$ in the boundary of its Vinberg domain $\Omega_f$. Then the preimage of $v$ under the map $\P(V)\setminus\P(\Span f)\to\P(V/\Span f)$ is a boundary face of $P$ whose closure contains $f$, which is a contradiction. To prove the converse, suppose $f$ is perfect. Then $P_f\subset\Omega_f$ by Fact \ref{fact:VinbergThm}, so $f$ is not properly contained in the closure of any boundary face of $P$.
\end{proof}

As a corollary, we have:

\begin{cor} \label{thm:maxiffpara}
    Let $P$ be a Coxeter polytope of negative type that has no proper negative type faces. Then the maximal boundary faces of $P$ are its parabolic faces.
\end{cor}

\section{Finite volume Coxeter polytopes}\label{sec:proofMainThm}

We can now state and prove our main theorem.

\begin{thm}\label{thm:main}
    Let $P$ be a Coxeter polytope of negative type. The following are equivalent:
    \begin{enumerate}[(i)]
        \item $P$ has finite volume in $\Omega_P$;
        \item $P$ has no proper negative type faces;
        \item $P$ is quasiperfect.
    \end{enumerate}
\end{thm}

\subsection{Facets of boundary faces}

Let $P$ be a Coxeter polytope of negative type and $f$ a boundary face of $P$ of dimension at least $1$.
Then $\overline f$ is a projective polytope in $\P(\Span f)$, and we let $\Facets(f)$ denote its facet set. Each $g\in\Facets(f)$ is also a face of $P$, so we can consider the set $S_g\subset S$ of facets of $P$ containing $g$.

\begin{rmk} \label{rmk:polytope f}
	Choose an element $s_g\in S_g\setminus S_f$ for each $g\in\Facets(f)$. Then $\overline f$ is the intersection of $\P(\Span f)$ with
	\begin{equation*}
    	\bigcap_{g\in\Facets(f)}\P\{v\in V\mid \alpha_{s_g}(v)\leq 0\}
	\end{equation*}
	In particular, the intersection of $\Span f$ and $\bigcap_{g\in\Facets(f)}\ker\alpha_{s_g}$ must be trivial.
\end{rmk}

\begin{prop} \label{prop:perpintersects}
    Let $P$ be a Coxeter polytope of negative type, and $f$ be either a parabolic face of $P$, or a proper negative type face of $P$.
	Then for any proper face $g$ of $f$ such that $S_g\cap S_f^\perp$ is of negative type (or empty), $g$ is a negative type face of $P$.
\end{prop}

\begin{proof}
	Let $T\coloneqq S_g\setminus (S_f\cup S_f^\perp) $, 
	\begin{equation*}
	S_1\coloneqq\bigcup\left\{\text{irreducible components $S'$ of $S_f \cup (S_f^\perp\cap S_g)$ such that $S'\subset T^\perp$}\right\}
    \end{equation*}
    and $S_2\coloneqq S_f\cup (S_f^\perp\cap S_g)\setminus S_1$.
    Note that $S_g=S_1\cup(S_2\cup T)$ and $S_2\cup T\subset S_1^\perp$. To get the desired conclusion, we show that
	\begin{inparaenum}[(i)]
		\item $S_2\cup T$ is of negative type, and
		\item $S_1$ is empty if $f$ is a parabolic face of $P$.
	\end{inparaenum}

	\begin{enumerate}[(i)]
		\item Let $I$ be an irreducible component of $S_2\cup T$. If $I\subset S_2$, then since $S_2$ is a union of irreducible components of $S_f\cup (S_f^\perp\cap S_g)$, $I$ is an irreducible component of $S_f\cup (S_f^\perp\cap S_g)$. Moreover, $I$ is also orthogonal to $(S_2\cup T)\setminus I\supset T$, so $I\subset S_1$ by definition of $S_1$, which is a contradiction. On the other hand, if $I\subset T$, then $I$ is orthogonal to $(S_2\cup T)\setminus I$, which contains $S_2$. Since $S_1$ and $S_2\cup T$ are orthogonal, it follows that $I$ is orthogonal to $S_1\cup S_2 \supset S_f$, so $I\subset S_f^\perp\cap S_g $, which contradicts $I\subset T$. Therefore we have both $I\cap S_2\neq\varnothing$ and $I\cap T\neq\varnothing$.

        We now claim that $I\cap S_2$ is a union of irreducible components of $S_2$. Let $J$ be an irreducible component of $S_2$ that intersects $I$. Since $I$ is an irreducible component of $S_2\cup T\supset J$, $J\cap I$ and $J\setminus I$ are orthogonal, and since $J$ is irreducible, $J\setminus I = \varnothing$, \ie $J\subset I$.

        In particular, $I\cap S_2$ has no irreducible components of positive type as $S_f$ is either of zero or negative type and $S_f^\perp\cap S_g$ is of negative type. Thus $W_{I\cap S_2}$ is a proper standard subgroup of $W_I$ that is not spherical, so $W_I$ is large by Fact \ref{fact:irredCoxgrp} and Remark \ref{rmk:irredCoxgrp}. Hence $I$ is of negative type by Fact \ref{fact:negTypeImpliesLargeOrATilda}.
		
		\item Suppose $f$ is a maximal boundary face of $P$ and $S_1\neq\varnothing$. Then $S_g=S_1\cup(S_2\cup T)$ is a nontrivial orthogonal decomposition and $S_2\cup T$ is of negative type by (i), so there exists a face $f'$ of $P$ such that $S_{f'}=S_1$ by Lemma \ref{lm:biggerFace}. Since $S_1\neq\varnothing$ is not of positive type and $S_2\cap T\neq\varnothing$, $f'$ is a boundary face of $P$ that properly contains $f$, which is a contradiction. Hence $S_1 = \varnothing$.
	\end{enumerate}
    Therefore, if $f$ is a parabolic face, then $S_g = S_2\cup T$ is of negative type by (i) and (ii). On the other hand, if $f$ is of negative type, then $S_1$ is also of negative type as a union of irreducible components of $S_f\cup (S_f^\perp\cap S_g)$, so $S_g = S_1 \cup (S_2\cup T)$ is also of negative type by (i).
\end{proof}

\subsection{Proof of Theorem \ref{thm:main}}
We start by showing the following implication.

\begin{prop} \label{prop:noNegTypeFaceImpliesQP}
    Let $P$ be a Coxeter polytope of negative type.
	If $P$ has no proper negative type faces, then it is quasiperfect.
\end{prop}

\begin{proof}
	Suppose $P$ has a parabolic face $f$ of positive dimension. Since $P$ has no proper negative type faces, Proposition \ref{prop:perpintersects} implies that $S_g\cap S_f^\perp\neq\varnothing$ for all $g\in\Facets(f)$. Hence we can choose a facet $s_g\in S_g\cap S_f^\perp$ for each $g\in\Facets(f)$.
	
	By Remark \ref{rmk:typeCartanMatrixGivesWitness}, we can find a vector $\Lambda = (\lambda_s)_{s\in S_f} >0$ such that $A_f\Lambda = 0$. If we let $x \coloneqq  \sum_{s\in S_f} \lambda_s v_s$, we see that $\alpha_s(x) = 0$ for all $s\in S_f$, \ie $x\in\Span f$. Notice also that if we let $\lambda_s = 0$ for $s\in S\setminus S_f$ and $\Lambda' = (\lambda_s)_{s\in S}\geq 0$, then $x = 0$ implies that $A_P\Lambda' = 0\geq 0$, so by Fact \ref{fact:CartanMatrixCharac}, we must have $\Lambda' = 0$, which is a contradiction.
	
	Thus $x$ is a nonzero element of $V_v\cap\Span f$. However, note that $V_v\subset\allowbreak\bigcap_{g\in \Facets(f)}\ker\alpha_{s_g}$ since $\{s_g\mid g\in\Facets_f\}\subset S_f^\perp$. Hence the intersection of $\Span f$ with $\bigcap_{g\in\Facets(f)}\ker\alpha_{s_g}$ is nonzero, contradicting Remark \ref{rmk:polytope f}.
	
	Therefore, $P$ has no parabolic faces of positive dimension. Since all non-positive type faces of $P$ are boundary faces, and are therefore contained in a parabolic face of $P$ by Proposition \ref{prop:maxbdryface}, we conclude that the proper faces of $P$ are either parabolic vertices, or elliptic. Hence, $P$ is quasiperfect.
\end{proof}

It is now elementary to conclude.

\begin{proof}[Proof of Theorem \ref{thm:main}]
    By Corollary \ref{cor:finvolimplies}, a finite volume Coxeter polytope of negative type $P$ has no proper faces of negative type. By Proposition \ref{prop:noNegTypeFaceImpliesQP}, a Coxeter polytope of negative type that has no proper negative type faces is quasiperfect. The last implication is given by \cite[\Thm 6.3]{marquis2017}, since a quasiperfect Coxeter polytope of negative type must be loxodromic irreducible by Fact \ref{fact:QPcharacterization}.
\end{proof}

\section{Minimality of the Vinberg domain}\label{sec:uniqueDomain}

We now set out to prove Theorem \ref{thm:uniqueDomain}. Along with Theorem \ref{thm:characVinDomainEqualsCHofLimSet} below, this provides a generalization of \cite[\Thm 8.2]{marquis2017}.
We start by finding when the candidates for the largest and the smallest properly convex domains are equal, \ie when the Vinberg domain equals the interior of the convex hull of the proximal limit set.

Recall that by convention, the convex hull of a subset $X$ of the closure of a properly convex domain $\Omega\subset\P(V)$ is taken to mean the convex hull of $X$ within any affine chart of $\P(V)$ containing $\Omega$. Similarly, “the” join of subsets of the closure of a properly convex domain $\Omega$ is the unique join of these subsets lying in $\overline\Omega$.

\begin{thm}\label{thm:characVinDomainEqualsCHofLimSet}
   Let $P$ be a negative type Coxeter polytope. The following are equivalent
   \begin{enumerate}[(i)]
      \item $\Omega_P = \Int\Conv(\Lambda_P)$;
      \item $P\subset\P(\Conv(v_s)_{s\in S})$;
      \item $P$ is a join of quasiperfect Coxeter polytopes of negative type.
   \end{enumerate}
\end{thm}

We will need a few additional lemmas.

\begin{lem}\label{lm:perfectImpliesPolInConvHullPolars}
   Let $P$ be a perfect Coxeter polytope of negative type. Then
   \begin{equation*}
      P\subset\P(\Conv(v_s)_{s\in S})
   \end{equation*}
\end{lem}
\begin{proof}
   Let $p$ be a vertex of $P$. Since $P$ is perfect, it is quasiperfect, so by Fact \ref{fact:QPcharacterization}, it is indecomposable. In particular, we can find a facet $s_p$ of $P$ that is not in $S_p\cup S_p^\perp$. Let $X = (\alpha_s(v_{s_p}))_{s\in S_p} \leq 0$. Notice that $X\neq 0$ as $s_p\not\in S_p^\perp$.

   Since $P$ is perfect, $p$ is elliptic, so $A_{S_p}$ is of positive type. By Fact \ref{fact:CartanMatrixCharac}, we can find a vector $\Lambda = (\lambda_s)_{s\in S_p}$ such that $A_{S_p}\Lambda = -X$ by splitting $S_p$ into its irreducible components. Hence $A_{S_p}\Lambda\geq 0$, so $\Lambda\geq 0$ by Fact \ref{fact:CartanMatrixCharac} again. If we let $x := v_{s_p} + \sum_{s\in S_p}\lambda_s v_s$, the identity $A_{S_p}\Lambda + X = 0$ implies that $\alpha_s(x) = 0$ for all $s\in S_p$. Thus $p = [x] \in \P(\Conv(v_s)_{s\in S})$.

   Since this holds for all the vertices of $P$, $P\subset\P(\Conv(v_s)_{s\in S})$.
\end{proof}
\begin{lem}\label{lm:negTypeFaceSubsetCHPolars}
   Let $P\subset\P(V)$ be a negative type Coxeter polytope such that 
   \begin{equation*}
      P\subset\P(\Conv(v_s)_{s\in S})
   \end{equation*}
   If $f$ is a proper negative type face of $P$, then
   \begin{equation*}
      f\subset\P(\Conv(v_s)_{s\in S_f^\perp})
   \end{equation*}
   If moreover $f$ is minimal, then $(\alpha_s,v_s)_{s\in S_f^\perp}$ is a quasiperfect Coxeter polytope structure of negative type on $f$ seen as a subset of $\P(\Span f)$.
\end{lem}
\begin{proof}
   Let $\Delta = \bigcap_{s\in S}\{\alpha_s\leq 0\}$ be the preferred lift of $P$. By Remark \ref{rmk:typeCartanMatrixGivesWitness} applied to the transpose of $A_{S_f}$, we can find a row vector $\Lambda =(\lambda_s)_{s\in S_f}$ with positive coordinates such that $\Lambda A_{S_f} < 0$. In other words, if we let $\alpha=\sum_{s\in S_f}\lambda_s\alpha_s$, then $\alpha(v_t) < 0$ for all $t\in S_f$. Therefore $\alpha(v_t)\leq 0$ for all $t\in S$ with equality if and only if $t\in S_f^\perp$. Since by assumption, $\Delta\subset\sum_{s\in S}\R_{\geq 0}v_s\subset\{\alpha\leq 0\}$, intersecting with $\{\alpha = 0\}$ yields
    \begin{equation*}
        \Delta\cap\Span f\subset\{\alpha = 0\}\cap\sum_{s\in S}\R_{\geq 0}v_s = \sum_{s\in S_f^\perp}\R_{\geq 0}v_s
    \end{equation*}
    as $\Delta\cap\{\alpha = 0\} = \Delta\cap \Span f$. In other words, $f\subset\P(\Conv(v_s)_{s\in S_f^\perp})$.

    Suppose now that $f$ is minimal, and pick any $g\in\Sigma(f)$. By Proposition \ref{prop:perpintersects}, if $S_g\cap S_f^\perp = \varnothing$, then $g$ is a smaller face of $P$ of negative type, contradicting the minimality of $f$. Hence the facets of $f$ are determined by the elements of $S_f^\perp$, so $(\alpha_s,v_s)_{s\in S_f^\perp}$ is a Coxeter polytope structure for $f\subset\P(\Span f)$. 
    Since $\Omega_P$ is properly convex and the Vinberg domain $\Omega_f\subset\P(\Span f)$ of $f$ is a subset of $\overline{\Omega_P}$, $\Omega_f$ is properly convex in $\P(\Span f)$, so $f$ is a Coxeter polytope of negative type by Fact \ref{fact:negTypeIffVinDomPropConv}.
    Moreover, observe that if $f$ (seen as a Coxeter polytope) had a proper face $g$ of negative type, then by definition, $S_g\cap S_f^\perp$ would be of negative type, so $g$ would also be a negative type face of $P$ by Proposition \ref{prop:perpintersects}, contradicting the minimality of $f$.
    Hence $f$ has no proper negative type faces, so it is quasiperfect by Theorem \ref{thm:main}.
\end{proof}

\begin{lem}\label{lm:polInConvHullPolarsImpliesJoinOfItsMinNegTypeFaces}
   Let $P$ be a Coxeter polytope of negative type such that
   \begin{equation*}
      P\subset\P(\Conv(v_s)_{s\in S})
   \end{equation*}
   Then $P$ is the join of its minimal negative type faces.
\end{lem}
\begin{proof}
   Let $f$ be a face of $P$ maximal for inclusion among the faces that are the join of minimal negative type faces. Since $P$ is of negative type, it always has at least one minimal negative type face, which is a (trivial) join of minimal negative type faces, so $f$ is well-defined. Suppose by contradiction that $f$ is a proper face of $P$. We will show that we can find a minimal negative type face $g$ of $P$ such that the join of $f$ and $g$ is a face of $P$, contradicting the maximality of $f$.

   Write $f$ as the join of minimal negative type faces $(f_i)_{i=1}^k$ of $P$.
   We consider
    \begin{equation*}
        Q:=\bigcap_{s\in S_f\cup S_f^\perp}\P\{\alpha_s\leq 0\}
    \end{equation*}
    Observe that $Q\supset P$ is non-empty and is properly convex as $\bigcap_{s\in S_f\cup S_f^\perp}\{\alpha_s = 0\} = \varnothing$ by applying Remark \ref{rmk:polytope f} to $f$. Indeed, the facets of $f$ are determined by the elements of $S_f^\perp$ since $f$ is the join of the $f_i$ and for each $i$, the facets of $f_i$ are given by $S_{f_i}^\perp\subset S_f^\perp$ by Lemma \ref{lm:negTypeFaceSubsetCHPolars}. Moreover, we can endow $Q$ with a Coxeter polytope structure with the data $(\alpha_s,v_s)_{s\in S_f\cup S_f^\perp}$.

    By Lemma \ref{lm:biggerFace} applied to the polytope $Q$ with $T_1 = S_f^\perp$ and $T_2 =S_f$, $S_f^\perp$ defines a face $f_Q^\perp$ of $Q$. 
    Since $\Span f = \bigcap_{s\in S_f}\ker\alpha_s$ and the facets of $f$ are determined by the elements of $S_f^\perp$, $f$ is also a face of $Q$. It follows that $Q$ is the join of $f$ and $f_Q^\perp$.
    
    Since $S_f$ is of negative type, we can find a vector $\Lambda := (\lambda_s)_{s\in S_f} > 0$ such that $A_{S_f}\Lambda < 0$ by Remark \ref{rmk:typeCartanMatrixGivesWitness}. If we let $x = \sum_{s\in S_f}\lambda_s v_s$, we then have $\alpha_t(x) \leq 0 $ for all $t\in S$ with equality if and only if $t\in S_f^\perp$. Hence $S_f^\perp$ also defines a face $f^\perp:= P\cap f_Q^\perp$ of $P$ such that $S_{f^\perp} = S_f^\perp$ by Definition \ref{def:definesAFace}.\eqref{item:definesAFaceSystem}. 

    We claim that $S_f^\perp = \bigsqcup_{i=1}^kS_{f_i}^\perp$. Observe that the $S_{f_i}^\perp$ are pairwise disjoint as $S_{f_i}^\perp\subset S_{f_j}$ whenever $j\neq i$, since $f$ is the join of the $f_i$ whose facets are given by $S_{f_i}^\perp$. Hence the right-hand side is indeed a disjoint union, and is clearly included in $S_f^\perp$. Moreover, the hyperplane spanned by an element of $S_f^\perp$ must define a facet of $f$ since it already contains $\P(\Span f_Q^\perp)$ and $Q$ is the join of $f$ and $f_Q^\perp$. Thus the reverse inclusion follows from the fact that the facets of $f$ are exactly given by the $S_{f_i}^\perp$.
    
    Now, since $f$ is the join of finitely many Coxeter polytopes of negative type whose Coxeter polytope structure is the restriction of the Coxeter polytope structure of $P$, $f$ is also a Coxeter polytope of negative type with the Coxeter structure given by the restriction of that of $P$ to $S_f^\perp$. 
    Therefore, $A_{S_f^\perp}$ is of negative type, so $f^\perp$ is a negative type face of $P$, and we can find a minimal negative type face $g$ of $P$ contained in $f^\perp$. 
    
    Note that $S_f^{\perp\perp} = S_f$ in $S$. This is because $s\in S_f^{\perp\perp}$ if and only if $\alpha_s(v_t) = 0$ for all $t\in S_f^\perp$, if and only if $f\subset\P(\ker\alpha_s)$ as $\Span f =\Span (v_s)_{s\in S_f^\perp}$ by Lemma \ref{lm:negTypeFaceSubsetCHPolars}, if and only if $s\in S_f$. It follows that $f^\perp = f_Q^\perp$, as $\Span f^\perp = \Span f_Q^\perp$ and $f^\perp$ is determined by the elements of $S_{f^\perp}^\perp = S_f^{\perp\perp} = S_f$ by Lemma \ref{lm:negTypeFaceSubsetCHPolars}.

    Therefore $Q$ is the join of $f$ and $f^\perp$, so the join $f\join g$ of $f$ and $g$ is a face of $Q$ contained in $P$, as both $f$ and $g$ are. Thus $f\join g$ is a face of $P$ since $P\subset Q$, larger than $f$, and that is the join of minimal negative type faces of $P$, which is a contradiction. Hence $f = P$, so $P$ is the join of its minimal negative type faces.
\end{proof}

We can now prove Theorem \ref{thm:characVinDomainEqualsCHofLimSet}.

\begin{proof}[Proof of Theorem \ref{thm:characVinDomainEqualsCHofLimSet}]
   Suppose first that $\Omega_P = \Int\Conv(\Lambda_P)$. By Fact \ref{fact:limSetInSpanPolars}, $\Lambda_P\subset\P(V_v)$, so $V_v = V$ and $\rho$ is reduced and dual-reduced.
   If $P$ is irreducible, $\Gamma_P$ is either large or affine of type $\Tilde{A}_d$ by Fact \ref{fact:negTypeImpliesLargeOrATilda}. If $\Gamma_P$ is affine of type $\Tilde{A}_d$, then it is perfect, so $P\subset\P(\Conv(v_s)_{s\in S})$ by Lemma \ref{lm:perfectImpliesPolInConvHullPolars}. If $\Gamma_P$ is large, Fact \ref{fact:DGKLMminDomain} and Remark \ref{rmk:largeDGKLMminDomain} imply that $P\subset P\cap\P(\Conv(v_s)_{s\in S})$, so $P\subset\P(\Conv(v_s)_{s\in S})$.

   If $P$ is reducible, then $P$ is decomposable by \cite[\Cor 4]{vinberg1971} (see \cite[\Thm 2.12]{marquis2017} for an easier formulation) since it is reduced and dual-reduced. Write $P$ as the join of irreducible Coxeter polytopes $(P_i)_{i=1}^k$ associated to the $\Gamma_P$-invariant decomposition $V = \bigoplus_{i=1}^k V_i$. 
   Since $P$ is of negative type, so are the $P_i$, and by Lemma \ref{lm:limSetOfDecCoxPol}, $\Gamma_P$ is proximal and $\Lambda_P = \bigcup_{i=1}^k\Lambda_{P_i}$. It follows that $\Omega_{P_i} = \Int\Conv(\Lambda_{P_i})$ for all $i$ since $\Omega_P = \Int\Conv(\Lambda_P)$, so by the irreducible case above, we have $P_i\subset\P(\Conv(v_s)_{s\in S_i})$ for all $i$. Since $S = \bigsqcup_{i=1}^k S_i$, this implies $P\subset\P(\Conv(v_s)_{s\in S})$. Hence (i) implies (ii).

   For (ii) implies (iii), suppose that $P\subset\P(\Conv(v_s)_{s\in S})$. By Lemmas \ref{lm:negTypeFaceSubsetCHPolars} and \ref{lm:polInConvHullPolarsImpliesJoinOfItsMinNegTypeFaces}, $P$ is the join of its minimal negative type faces, which are quasiperfect Coxeter polytopes of negative type.

   Lastly, for (iii) implies (i), write $P$ as a join of the quasiperfect Coxeter polytopes of negative type $(P_i)_{i=1}^k$. By \cite[\Thm 8.2]{marquis2017} and its proof, we have $\Omega_{P_i} = \Int\Conv(\Lambda_{P_i})$ for all $i$. Since $\Omega_P$ is a join of the $\Omega_{P_i}$, and $\Lambda_P$ the union of the $\Lambda_{P_i}$ by Lemma \ref{lm:limSetOfDecCoxPol}, it follows that $\Omega_P = \Int\Conv(\Lambda_P)$.
\end{proof}

The proof of Theorem \ref{thm:uniqueDomain} follows.

\begin{proof}[Proof of Theorem \ref{thm:uniqueDomain}]
   If $P$ is a quasiperfect Coxeter polytope such that $\dim P\geq 2$, we have $\card{S}\geq 3$. Since it is also irreducible by Fact \ref{fact:QPcharacterization}, and hence reduced and dual-reduced by Fact \ref{fact:irredIffRedAndDualRed}, Fact \ref{fact:DGKLMmaxDomain} ensures that $\Omega_P$ is the largest $\Gamma_P$-invariant properly convex domain in $\P(V)$. In particular, any $\Gamma_P$-invariant properly convex domain lies in $\Omega_P$, so by Fact \ref{fact:limitset}, it must contain $\Int\Conv(\Lambda_P)$. Thus $\Int\Conv(\Lambda_P)$ is in fact the smallest $\Gamma_P$-invariant properly convex domain of $\P(V)$. By Theorem \ref{thm:characVinDomainEqualsCHofLimSet}, it equals $\Omega_P$, so $\Omega_P$ is the unique $\Gamma_P$-invariant properly convex domain.
    
    On the other hand, suppose that $\Omega_P$ is the unique $\Gamma_P$-invariant properly convex domain in $\P(V)$. In particular, we must have $\Int\Conv(\Lambda_P) = \Omega_P$, so Theorem \ref{thm:characVinDomainEqualsCHofLimSet} implies that $P$ is a join of quasiperfect Coxeter polytopes of negative type. However, if $P$ is a non-trivial join, then so is $\Omega_P$, and by the discussion following paragraph \ref{par:decPolytopes}, there must be at least one more $\Gamma_P$-invariant properly convex domain. Hence $P$ is a quasiperfect Coxeter polytope of negative type.
    
    Suppose that $\dim P \leq 1$. Since $P$ is of negative type, we must have that $\dim P = 1$, so $\Omega_P$ is a segment in $\mathbb{RP}^1$. However in this case, $\Gamma_P$ also preserves the properly convex complement $\mathbb{RP}^1\setminus\Omega_P$, which is a contradiction. Hence $\dim P\geq 2$.
\end{proof}

The above discussion is linked to the question of whether or not the proximal limit set of a group $\Gamma<\SLpm(V)$ that quasidivides a properly convex domain $\Omega\subset\P(V)$ is equal to the boundary of the domain. The first result in this direction is in the case where $\Gamma$ divides $\Omega$ indecomposable non-symmetric, where Blayac showed in \cite{blayac2024} that indeed, $\Lambda_\Gamma = \partial\Omega$. 

Theorem \ref{thm:characVinDomainEqualsCHofLimSet} can be seen as a first step in proving the result for irreducible reflection groups.
Indeed, by Fact \ref{fact:negTypeImpliesLargeOrATilda} and the fact that if $P$ is a Coxeter polytope of negative type and $W_P$ is affine of type $\Tilde{A}_d$, then $\Omega_P$ is a simplex (see \cite{margulisVinberg2000}), we have

\begin{cor}\label{cor:ifProxLimSetFillsThenQPLarge}
    Let $P$ be a Coxeter polytope of negative type such that $\Lambda_P = \partial\Omega_P$. Then $P$ is large quasiperfect.
\end{cor}

We conjecture that the reciprocal holds as well.

\begin{conj}\label{conj:characProxLimSetFilling}
    Let $P$ be a Coxeter polytope of negative type. Then $\partial\Omega_P = \Lambda_P$ if and only if $P$ is large quasiperfect.
\end{conj}

Observe that by Theorem \ref{thm:characVinDomainEqualsCHofLimSet}, the missing implication would be a consequence of

\begin{conj}
    Let $P$ be a quasiperfect Coxeter polytope of negative type. Then the extreme points of $\overline{\Omega_P}$ are dense in $\partial\Omega_P$ unless $W_P$ is affine of type $\Tilde{A}_d$.
\end{conj}

\printbibliography

\end{document}